\documentclass[10pt]{article}

\usepackage[utf8]{inputenc}
\usepackage{graphicx}
\usepackage{cite}
\usepackage{amsmath,amsfonts,amsthm,url,color,amssymb}
\usepackage{verbatim}
\usepackage{enumerate}
\usepackage{graphicx}
\usepackage{url}
\usepackage{tikz}
\usepackage{authblk}
\newcommand{\C}{\mathcal C}

\newcommand{\F}{\mathbb{F}}

\newtheorem{theorem}{Theorem}[section]

\newtheorem{proposition}[theorem]{Proposition}
\newtheorem{definition}[theorem]{Definition}
\newtheorem{problem}{Problem}
\newtheorem{lemma}[theorem]{Lemma}
\newtheorem{corollary}[theorem]{Corollary}

\newtheorem{remark}[theorem]{Remark}

\begin{document}

\title{Character sums over affine spaces and applications\\}

\author{Lucas Reis}

\affil {Departamento de Matem\'{a}tica,
Universidade Federal de Minas Gerais,
UFMG,
Belo Horizonte MG (Brazil),
 30123-970  \\
 email: lucasreismat@mat.ufmg.br}
\date{\today}

\maketitle

\begin{abstract}
Given a finite field $\F_q$, a positive integer $n$ and an $\F_q$-affine space $\mathcal A\subseteq \mathbb \F_{q^n}$, we provide a new bound on the sum 
$\sum_{a\in \mathcal A}\chi(a)$,
where $\chi$ a multiplicative character of $\F_{q^n}$. We focus on the applicability of our estimate to results regarding the existence of special primitive elements in $\F_{q^n}$. In particular, we obtain substantial improvements on previous works.
\end{abstract}

\section{Introduction}
Let $q=p^k$ be a prime power and $\F_q$ be the finite field with $q$ elements. A multiplicative character (resp. additive character) $\chi$ of $\F_{q}$ is an homomorphism from the group $(\F_{q}^*, \times)$ (resp. $(\F_q, +))$ to the unit circle. Characters of finite fields is a valuable tool in Number Theory, Combinatorics and Arithmetic Dynamics, being employed in the proof of results on the existence and distribution of numerous algebraic and combinatorial objects. See~\cite{char} for a rich source of results, problems and techniques in the subject. 

When proving results on existence and distribution with the help of characters,  it is frequently required to provide estimates to sums $$s(\chi, S):=\sum_{s\in S}\chi(s), S\subseteq \F_q.$$ One of the earliest works regarding bounds on character sums is due to Polya and Vinogradov. They considered $q=p$ a prime and $S$ an arbitrary interval of integers, obtaining the famous bound $|s(\chi, S)|\ll \sqrt{p}\log p$ for $\chi$ a  non trivial multiplicative character. This bound was further improved by Burgess~\cite{bu3} to $\ll p^{-\delta(\varepsilon)}|S|$, under the condition $|S|>p^{1/4+\varepsilon}$. The ideas of Burgess were extended to non trivial multiplicative characters of non prime fields, still in the spirit of its early work. More specifically, for a given basis $\{b_1, \ldots, b_k\}$ of the $\F_p$-vector space $\F_{p^k}$, the set $S$ comprises elements $\sum_{i=1}^ka_ib_i$, with the integers $a_i$ lying in non degenerate intervals of integers. For more details, see~\cite{chang2, dl} and the references therein. We observe that such results cannot be applied when $S$ is a generic $\F_p$-vector space of $\F_{q}$. For some special vector spaces, non trivial estimates were provided. If $\theta\in \F_{q}$ generates $\F_q$ over $\F_p$, i.e., $\F_q=\F_p(\theta)$, and 
$$S=\bigoplus _{i=0}^{m-1}\theta^i\cdot \F_q:=\left\{\sum_{i=0}^ma_i\theta^i\,|\, a_i\in \F_p\right\}, m\le k,$$
Burgess~\cite{bu6} proved that $|s(\chi, S)|=O(p^{m(1-\delta(\varepsilon))})$ if $m>k(1/4+\varepsilon)$. For a ``sufficiently'' generic $\F_p$-vector space $S\subseteq \F_q$ of dimension $m>\rho k$ with $\rho>0$, Chang~\cite{chang} proved that $|s(\chi, S)|\le |S|\cdot (\log p)^{-\delta}$ under $n\ll p(\log p)^{-4}$. In the same paper, she proposes a problem on bounding $|s(\chi, S)|$ non trivially for a generic $\F_p$-vector space $S$; see Problem 2~in~\cite{chang}.

For each positive integer $n$, let $\F_{q^n}$ be the unique $n$-degree extension of $\F_q$ and let $V$ an $\F_q$-vector space of dimension $t$. From the well-known Weil bound, we can easily show that $|s(\chi, V)|\le q^{n/2}$ if $\chi$ is a non trivial multiplicative character; see Lemma~\ref{lem:trivia}. This bound becomes trivial if $t\le n/2$, limiting the generality of existence and distribution results where a character sum method is employed. Some interesting examples concern {\em primitive elements} of $\F_{q^n}$, i.e., generators of the cyclic group $\F_{q^n}^*$. Primitive elements with specified properties have been extensively studied (see Chapter~3 of~\cite{char} for classical problems and results). This is mainly motivated by the fact that primitive elements are useful in applications such as Cryptography~\cite{dh}. In~\cite{dig4}, the existence of primitive elements with prescribed {\em digits}~\cite{dig2} is asymptotically proved up to around {\em half } of the digits. In~\cite{R} the author explores variations of the Primitive Normal Basis Theorem~\cite{lenstra}, providing results on the existence of primitive elements in $\F_{q^n}$ whose $\F_q$-Galois conjugates generate an $\F_q$-vector space of dimension $n-k$, where $k\in [0, n(1/2-\varepsilon_{q, n})]$. However, it is unclear whether such results are sharp or not, i.e., we do not have ``if and only if'' conditions in their statements. 

In this paper we discuss character sums estimates over $\F_q$-affine spaces $\mathcal A$ in finite fields. We usually write $\mathcal A=u+\mathcal V$, where $\mathcal V$ is an $\F_q$-vector space. Our main result, Theorems~\ref{thm:main} and~\ref{thm:Main} provide a new upper bound on $|s(\chi, \mathcal A)|$ in the range $t\le n/2$, where $t$ is the dimension of $\mathcal A$ over $\F_q$. This bound is non trivial when $q>n^2$. Although such condition might look restrictive, it can be very powerful in providing existence results. This is nicely exemplified in the following theorem that characterizes the affine spaces over large fields that contain primitive elements.
\begin{theorem}\label{thm:gold}
For each $n\ge 2$ there exists $c(n)>0$ such that, for every $q>c(n)$, an $\F_q$-affine space $\mathcal A=u+\mathcal V\subseteq \F_{q^n}$ of dimension $t\ge 1$ contains a primitive element of $\F_{q^n}$ if and only if one of the following holds:
\begin{enumerate}[(i)]
\item there exists a primitive element $y\in \F_{q^n}$ and divisor $d<n$ of $n$ such that $$y\in \mathcal A\subseteq y\cdot \F_{q^d};$$
\item there exist nonzero elements $y\in\mathcal A$ and $z\in\mathcal V$ such that $\F_{q^n}=\F_q(yz^{-1})$.
\end{enumerate}
\end{theorem}
Theorem~\ref{thm:gold} can be viewed as a generalization of works of Carlitz~\cite{C} and Davenport~\cite{D} regarding the extensions $\F_{q^n}/\F_q$ with the {\em translate property}. These are the extensions such that, for every $\theta\in \F_{q^n}^*$ with $\F_q(\theta)=\F_{q^n}$, 
there exists a primitive element in the affine space $\theta+\F_q$. They prove that there exists a constant $c_n>0$ such that $\F_{q^n}/\F_q$ possesses the translate property for every $q>c_n$.

Further applications of Theorem~\ref{thm:Main} provide asymptotically sharp results on the problems considered in~\cite{dig4, R}. In the context of~\cite{dig4}, we show that we can prescribe $n(1-1/r)-1\ge n/2-1$ of the digits of primitive elements in $\F_{q^n}$, where $r$ is the smallest prime factor of $n$. We also prove that this number of prescribed digits is optimal. On the work in~\cite{R}, we improve the range $ [0, n(1/2-\varepsilon_{q, n}))$ there to the best possible $[0, n-2]$, obtaining an ``if and only if'' theorem; this is a major step towards the solution of a problem proposed in~\cite{HMPT}. We also obtain a seemingly counter intuitive configuration of primitive elements that arises from a conjecture on Artin-Scherier extensions $\F_{p^p}$.

The structure of the paper is given as follows. Section 2 provides background material that is used along the way. In Section 3 we provide our main results on character sums estimates. In Section 4 some applications of our main results are given.
\section{Preparation}
This section provides some auxiliary results that are further required. We start with a basic definition.

\begin{definition}
Let $\overline{\F}_q$ be the algebraic closure of $\F_q$. An element $\alpha\in \overline{\F}_q$ has degree $k$ over $\F_q$ if one of the following (equivalent) statements hold:
\begin{enumerate}[(i)]
\item $\alpha$ generates $\F_{q^k}$ over $\F_q$, i.e., $\F_{q^k}=\F_q(\alpha)$;
\item the degree of the minimal polynomial of $\alpha$ over $\F_q$ equals $k$;
\item $\F_{q^k}$ is the smallest extension of $\F_q$ that contains $\alpha$.
\end{enumerate}
\end{definition}

We have the following result.

\begin{proposition}\label{prop:crucial}
Fix $q$ a prime power and $n\le q$. For any $\F_q$-affine space $\mathcal A=u+\mathcal V\subseteq \F_{q^n}$ of dimension $t\ge 1$, and any nonzero $y\in \mathcal A$, one of the following holds:

\begin{enumerate}[(i)]
\item there exists a divisor $d$ of $n$ such that $d<n$ and $y\in \mathcal A\subseteq y\cdot \F_{q^d}$;
\item there exists nonzero $z\in \mathcal V$ such that $\F_{q^n}=\F_q(yz^{-1})$.
\end{enumerate}
\end{proposition}

\begin{proof}
 For each divisor $d$ of $n$, let $\mathcal C_{d}$ be the set of nonzero elements $v\in \mathcal V$ such that $yv^{-1}\in \F_{q^d}$ and let $\delta_{d}=|\mathcal C_{d}|$. We observe that $yv^{-1}\in \F_{q^d}$ if and only if $a(yv^{-1})^{-1}=a\cdot y^{-1}v\in \F_{q^d}$ for every $a\in \F_q$. The latter implies that $\mathcal C_{d}\cup\{0\}$ is an $\F_q$-vector space. If its dimension equals $t$, the dimension of $\mathcal V$, it follows that $\mathcal V\subseteq y\cdot \F_{q^d}$. In this case, since $y\in \mathcal A$, we have that $y\in \mathcal A\subseteq y\cdot \F_{q^d}$. Otherwise, if $p_1, \ldots, p_s$ are the prime divisors of $n$, we have that $\delta_{n/p_i}\le q^{t-1}$ and so  
$$\delta_{n}\ge |V|-1-\sum_{i=1}^s\delta_{n/p_i}\ge q^t-1-s\cdot q^{t-1}>0,$$
since $s<n\le q$. 

\end{proof}

\subsection{Multiplicative characters and primitivity}

We provide some basic facts on multiplicative characters of finite fields, including a character sum formula for the characteristic function of primitive elements in $\F_{q^n}$. Fix $\theta$ a primitive element in $\F_{q^n}$, i.e., $\langle \theta\rangle=\F_{q^n}^*$. The set $\widehat{\F_{q^n}}$ of multiplicative characters of $\F_{q^n}$ comprises the characters $\chi_{\theta, k}, 0\le k\le q^n-2$ defined as follows: if $a=\theta^{d}$ with $0\le d\le q^n-2$, we set
$$\chi_{\theta, k}(a)=\mathbf{e}\left(\frac{2\pi i (kd)}{q^n-1}\right),$$
where $\mathbf{e}(z):=e^{z }$ is the complex exponential function. We fix $\theta$ and simply write $\chi_{\theta, k}=\chi_k$. The set $\widehat{\F_{q^n}}$ is a cyclic multiplicative group of order $q^n-1$, hence isomorphic to $\F_{q^n}^*$. The identity element is the trivial character $\chi_0$, mapping any nonzero element $a\in \F_{q^n}$ to $1\in\mathbb C$. We naturally extend the multiplicative characters to $0\in \F_{q^n}$ by setting $\chi(0)=0$ for every $\chi\in \widehat{\F_{q^n}}$. We observe that, with our notation, the character $\chi_k$ has order $\frac{q^n-1}{\gcd(q^n-1, k)}$. In particular, for each divisor $d$ of $q^n-1$, there exist $\varphi(d)$ multiplicative characters of order $d$. For each divisor $d$ of $q^n-1$, let $$\Lambda(d)=\left\{\chi_{\frac{(q-1)e}{d}}\,|\, 0\le e\le d,\, \gcd(e, d)=1 \right\},$$ be the set of elements in $\widehat{\F_{q^n}}$ of order $d$. We have the following well-known result.

\begin{lemma}[\cite{HMPT}, Section 5.2]\label{lem:func}
Let $\mu$ and $\varphi$ denote the Moebius and Euler totient functions over the integers, respectively. Then,  for each $y\in \F_{q^n}$, the sum
$$\frac{\varphi(q^n-1)}{q^n-1}\sum_{d|q^n-1}\frac{\mu(d)}{\varphi(d)}\sum_{\chi\in \Lambda(d)}\chi(y),$$
equals $1$ if $y$ is a primitive element of $\F_{q^n}$ and equals $0$, otherwise.
\end{lemma}

\subsection{Character sum bounds}
Here we provide some known character sum estimates that are of our interest.
\begin{theorem}[\cite{LN}, Theorem 5.41]\label{Gauss1} Let $\eta$ be a multiplicative character of $\F_{q^n}$ of order $r>1$ and $F\in \F_{q^n}[x]$ be a polynomial of positive degree such that $F$ is not of the form $ag(x)^r$ for some $g\in \F_{q^n}[x]$ with degree at least $1$ and $a\in \F_{q^n}$. Suppose that $z$ is the number of distinct roots of $F$ in its splitting field over $\F_{q}$. Then the following holds: 
$$\left|\sum_{c\in \F_{q^n}}\eta(F(c))\right|\le (z-1)\sqrt{q}.$$
\end{theorem}
As an application of the previous theorem, we have the following result.

\begin{lemma}\label{lem:trivia}
Let $\mathcal A\subseteq \F_{q^n}$ be an $\F_q$-affine space of dimension $t>0$ and $\chi$ a non trivial multiplicative character of $\F_{q^n}$. Then
$$\left|\sum_{a\in \mathcal A}\chi(a)\right|\le q^{n/2}.$$
\end{lemma}

\begin{proof}
From Lemma~3.4 in~\cite{R2}, there exists a separable polynomial $f\in \F_q[x]$ of degree $q^{n-t}$ such that $f(\F_{q^n})=\mathcal A$ and, for each $a\in \mathcal A$, the equation $f(x)=a$ has exactly $q^{n-t}$ solutions in $\F_{q^n}$. Therefore,
$$\sum_{a\in \mathcal A}\chi(a)=\frac{1}{q^{n-t}}\sum_{x\in \F_{q^n}}f(x).$$
From construction, the polynomial $f(x)$ is separable, hence is not of the form $ag(x)^r$ for some $g\in \F_{q^n}[x]$ with degree at least $1$ and $a\in \F_{q^n}$. Since $\chi$ is non trivial, Theorem~\ref{Gauss1} entails that $\left|\sum_{x\in \F_{q^n}}\chi(g(x))\right|\le q^{3n/2-t}$, from where the result follows.
\end{proof}

The following character sum estimate, due to Katz~\cite{K}, is crucial in this paper.

\begin{theorem}[Katz]\label{thm:katz}
Let $\theta$ be an element of degree $n$ over $\F_q$ and $\chi$ a non trivial multiplicative character of $\F_{q^n}$. Then 
$$\left|\sum_{a\in \F_q}\chi(\theta+a)\right|\le (n-1)\sqrt{q}.$$
\end{theorem}

\section{Main result}
The main result of this paper is the following theorem.

\begin{theorem}\label{thm:main}
Let $\mathcal A\subseteq \mathbb \F_{q^n}$ be an $\F_q$-affine space of dimension $t\ge 1$, where $n>1$. For each divisor $d$ of $n$, let $n_{\mathcal A, d}$ be the number of elements in $\mathcal A$ whose degree over $\F_q$ equals $d$. If $\chi$ is a nontrivial multiplicative character of $\F_{q^n}$, then 
\begin{equation}\label{eq:bound1}\left|\sum_{b\in \F_q}\sum_{a\in \mathcal A}\chi(a+b)\right|\le \sum_{d|n}n_{\mathcal A, d}\cdot \delta_{\chi, d},\end{equation}
where $\delta_{\chi, d}=q$ if $\chi |_{\F_{q^d}}$ is trivial and $\delta_{\chi, d}=\min\{q, (d-1)\sqrt{q}\}$, otherwise. In particular, if $n_{\mathcal A, n}>0$, we have that
\begin{equation}\label{eq:bound}\left|\sum_{b\in \F_q}\sum_{a\in \mathcal A}\chi(a+b)\right|< n\cdot q^{t+1/2}.\end{equation}
\end{theorem}
\begin{proof}
For each divisor $d$ of $n$, let $C_{\mathcal A, d}$ be the set of elements $a\in \mathcal A$ with degree $d$ over $\F_q$. We have that
$$\left|\sum_{b\in \F_q}\sum_{a\in \mathcal A}\chi(a+b)\right|\le \sum_{d|n}\sum_{a\in \C_{\mathcal A, d}}\left|\sum_{b\in \F_q}\chi(a+b)\right|.$$
Employing Theorem~\ref{thm:katz}, we obtain Eq.~\eqref{eq:bound1}. We proceed to Eq.~\eqref{eq:bound}. For each divisor $d$ of $n$ with $d<n$, let $\Delta_{d}$ be the set of elements in $\mathcal A$ whose degree over $\F_q$ is a divisor of $d$. In particular, $u-v\in \F_{q^d}$ whenever $u, v\in \Delta_d$ and so the set $\Delta_d$ is of the form $t+V_d$, where $V_d$ is an $\F_q$-vector space contained in $\F_{q^d}$. Therefore, if $n_{\mathcal A, n}>0$, we have that $V_d$ has dimension at most $t-1$. In this case, if $p_1, \ldots, p_s$ are the distinct prime divisors of $n$, we have that $|\Delta_{n/p_i}|\le q^{t-1}$ and, from construction,
$$\bigcup_{i=1}^s\Delta_{n/p_i}=\bigcup_{d|n\atop d<n}C_{\mathcal A, d}.$$
In conclusion, $\Delta:=\sum_{d|n\atop d<n}n_{\mathcal A, d}\le s\cdot q^{t-1}$. Applying estimates to Eq.~\eqref{eq:bound1}, we obtain
\begin{align*}\left|\sum_{b\in \F_q}\sum_{a\in \mathcal A}\chi(a+b)\right| & \le (q^t-\Delta)(n-1)\sqrt{q}+\Delta\cdot q\\ {}&=q^{t+1/2}(n-1)+\Delta\cdot \sqrt{q}(\sqrt{q}-n+1)\\ {} & \le q^{t+1/2}(n-1)+sq^{t-1/2}(\sqrt{q}-n+1)\\ {} & <n\cdot q^{t+1/2},\end{align*}
for $n\le \sqrt{q}+1$ (observe that $s<n$). If $n>\sqrt{q}$,~Eq.~\eqref{eq:bound} is trivial since there are $q^{t+1}$ terms in the sum, each a complex number of norm $1$.
\end{proof}

The following theorem provides an alternative form of Theorem~\ref{thm:main}, which can be more useful for applications.

\begin{theorem}\label{thm:Main}
Let $\mathcal A=u+\mathcal V\subseteq\mathbb \F_{q^n}$ be an $\F_q$-affine space of dimension $t\ge 1$, where $n>1$. Suppose that there exists a nonzero element $y\in \mathcal V$ such that the set $\mathcal A_y:=\{ay^{-1}\,|\, a\in \mathcal A\}$ contains an element of degree $n$ over $\F_q$. If $\chi$ is a non trivial multiplicative character of $\F_{q^n}$ we have that 
\begin{equation}\label{eq:main}\left|\sum_{a\in \mathcal A}\chi(a)\right|< n\cdot q^{t-1/2}.\end{equation}
\end{theorem}

\begin{proof}
Since $\chi$ is multiplicative, $\left|\sum_{a\in \mathcal A}\chi(a)\right|=\left|\sum_{a\in \mathcal A_y}\chi(a)\right|$. By the definition we have that $\mathcal A_y=u_y+V_y$, where $V_{y}$ is an $\F_q$-vector space of dimension $t$ containing $\F_q$. Therefore, we necessarily have a decomposition $\mathcal A_y=\F_q\oplus \mathcal B$, where $\mathcal B$ is an $\F_q$-affine space of dimension $t-1$. In particular, the following holds:
$$\sum_{a\in \mathcal A_y}\chi(a)=\sum_{b\in \F_q}\sum_{a\in \mathcal B}\chi(a+b).$$
From hypothesis, $\mathcal B$ contains an element whose degree over $\F_q$ equals $n$, and so Eq.~\eqref{eq:main} follows by Eq.~\eqref{eq:bound} in Theorem~\ref{thm:main}.
\end{proof}

\begin{remark}
We observe that Theorem~\ref{thm:Main} is trivial if $n>\sqrt{q}$. In the range $n\le \sqrt{q}$, Proposition~\ref{prop:crucial} entails there is no loss of generality by assuming that $\mathcal A$ satisfies the condition required in Theorem~\ref{thm:Main}. In fact,  if $\mathcal A\subseteq y\cdot \F_{q^d}$ for some proper divisor $d$ of $n$ and some nonzero $y\in\mathcal A$, any character sum estimate over $\mathcal A$ reduces to a character sum estimate over another $\F_q$-affine space $\mathcal A_0\subseteq \F_{q^d}$.
\end{remark}

%The following corollary provides a general instance where Eq.~\eqref{eq:main} holds.
%\begin{corollary}
%Let $r$ be the smallest prime divisor of $n>1$ and let $\mathcal A\subseteq\mathbb \F_{q^n}$ be an $\F_q$-affine space of dimension $t>n/r$. If $\chi$ is a non trivial multiplicative character of $\F_{q^n}$ and $n\le \sqrt{q}+1$, we have that 
%$\left|\sum_{a\in \mathcal A}\chi(a)\right|\le n\cdot q^{t-1/2}$.
%\end{corollary}
%\begin{proof}
%We observe that any intermediate field of the extension $\F_{q^n}/\F_{q}$ is of the form $\F_{q^d}$ with $d\le n/r$. In particular, the number of elements in $\mathcal A$ whose degree over $\F_q$ equals $n$ is at least
%$q^t-\sum_{i=1}^{n/r}q^i=q^t-\frac{q^{n/r}-1}{q-1}$.
%The latter is positive provided that $t>n/r$. The result follows by Theorem~\ref{thm:Main}.
%\end{proof}

\section{Application: primitive elements in special configurations}
In this section we provide some applications of our main results. We consider them separately into subsections. We start with the proof of Theorem~\ref{thm:gold}.

\subsection{Proof of Theorem~\ref{thm:gold}}
Taking $q\ge n$, Proposition~\ref{prop:crucial} implies that $\mathcal A$ satisfies item (i) or (ii) in Theorem~\ref{thm:gold} whenever $\mathcal A$ contains a primitive element. The other direction follows by the following proposition.

\begin{proposition}\label{prop:main}
Let $n\ge 1$ be an integer and let $\mathcal A=u+\mathcal V\subseteq \F_{q^n}$ be an $\F_q$-affine space of dimension $t\ge 1$. Suppose that there exist nonzero elements $y\in\mathcal A$ and $z\in\mathcal V$ such that $yz^{-1}$ has degree $n$ over $\F_q$. Then the number $\mathcal P(\mathcal A)$ of primitive elements in $\mathcal A$ satisfies the following inequality:
\begin{equation*}\label{eq:sieve}\mathcal P(\mathcal A)> q^t\cdot \frac{\varphi(q^n-1)}{q^n-1}\left(1-\frac{n\cdot W(q^n-1)}{\sqrt{q}}\right),\end{equation*}
where $W(q^n-1)$ denotes the number of squarefree divisors of $q^n-1$. In particular if $n$ is fixed and $\varepsilon>0$ is arbitrary, there exists $c=c(\varepsilon)>0$ such that, for $q>c(\varepsilon)$ we have that $$\mathcal P(\mathcal A)\ge q^{t-\varepsilon}.$$
\end{proposition}

\begin{proof}
In the notation of Lemma~\ref{lem:func}, we have that
\begin{align}\begin{aligned}\label{eq:pt}\frac{(q^n-1)\cdot \mathcal P(\mathcal A)}{\varphi(q^n-1)}= & \sum_{a\in\mathcal A}\sum_{d|q^n-1}\frac{\mu(d)}{\varphi(d)}\sum_{\chi\in \Lambda(d)}\chi(a)\\ = &\sum_{a\in\mathcal A}\chi_0(a)+\sum_{d|q^n-1\atop d\ne 1}\frac{\mu(d)}{\varphi(d)}\sum_{\chi\in \Lambda(d)}\sum_{a\in\mathcal A}\chi(a),\end{aligned}\end{align}
where $\chi_0$ is the trivial multiplicative character. Since $\mathcal A$ has at most one zero element, it follows that $\sum_{a\in\mathcal A}\chi_0(a)\ge q^t-1$. For $d\ne 1$, every character $\chi\in \Lambda(d)$ is non trivial. From hypothesis, $\mathcal A$ satisfies the conditions of Theorem~\ref{thm:Main} and so $\left|\sum_{a\in\mathcal A}\chi(y)\right|\le nq^{t-1/2}$ for every $\chi\in \Lambda(d), d\ne 1$. Applying estimates to Eq.~\eqref{eq:pt}, we obtain that
\begin{align*}\frac{(q^n-1)\cdot \mathcal P(\mathcal A)}{\varphi(q^n-1)}& \ge  q^t-1-nq^{t-1/2}\sum_{d|q^n-1\atop d>1, \mu(d)\ne 0}\frac{1}{\varphi(d)}\sum_{\chi\in\Lambda(d)}1\\ {} & > q^t-n\cdot W(q^n-1)\cdot q^{t-1/2}.\end{align*}
The bound $\frac{(q^n-1)\cdot \mathcal P(\mathcal A)}{\varphi(q^n-1)}\ge q^{t-\varepsilon}$ under $q>c(\varepsilon)$ follows from the well-known bounds $W(\ell)=\ell^{o(1)}$ and $\varphi(\ell)= \ell^{1-o(1)}$. 
\end{proof}

\subsection{Grassmannians avoiding primitive elements}
For integers $n\ge 1$ and $1\le k\le n$, the Grassmannian $\mathcal G(n, k)$ is the set of all $k$-dimensional $\F_q$-vector spaces in $\F_{q^n}$. Considering the context of this paper, it is natural to to ask what is the greatest integer $t=t(n, q)\le n$ such that $\mathcal G(n, t)$ contains an element $V$ free of primitive elements. Equivalently, $t(n, q)$ is the unique integer such that for every integer $t(n, q)<k\le n$, each element $V\in \mathcal G(n, k)$ contains a primitive element. As an application of Theorem~\ref{thm:gold}, we provide a formula for $t(n, q)$ when $q$  is sufficiently large.

\begin{proposition}\label{prop:grass}
For each positive integer $n\ge 2$, let $p_n$ be its smallest prime factor. Then there exists a constant $c(n)>0$ such that, for every $q>c(n)$, we have that
$$t(n, q)=n/{p_n}.$$
\end{proposition}

\begin{proof}
We observe that $\F_{q^{n/p_n}}\in \mathcal G(n, n/p_n)$ is a subfield of $\F_{q^n}$, hence it cannot contain primitive elements. In particular, $t(n, q)\ge n/p_n$. Pick an integer $k> n/p_n$ and let $V\in \mathcal G(n, k)$. We claim that $V$ contains two nonzero elements $y, z$ such that the degree of $yz^{-1}$ equals $n$. By Theorem~\ref{thm:gold}, this proves the result. Let $z$ be a nonzero element of $V$ and, for each divisor $d$ of $n$, let $\delta_d$ be the number of elements $vz^{-1}, v\in V$ with degree $d$ over $\F_q$. We have the trivial bound 
$$\sum_{d|n\atop d<n}\delta_d\le \sum_{s=1}^{n/p_n}q^i<q^{n/p_n+1}\le |V|,$$
and so $\delta_n>0$. 
\end{proof}

\subsubsection{Primitive elements and digits}
Motivated by works of Mauduit and Rivat~\cite{MR, MR2} on the famous Gelfond Problems about digits over the integers, Dartyge and Sarkozy~\cite{dig2} introduced the notion of digits over finite fields. If $\mathcal B=\{b_1, \ldots, b_k\}$ is a basis for $\F_{q^n}$, regarded as an $\F_q$-vector space, then every $y\in \F_{q^n}$ is written uniquely as $\sum_{i=1}^na_ib_i$, where $a_i\in \F_q$. The elements $a_1, \ldots, a_n$ are called the {\em digits} of $y$ with respect to the basis $\mathcal B$. In~\cite{dig2} the authors explore the existence and number of polynomial values $P(z)$ with prescribed sum of digits with respect to arbitrary basis, where $P$ is a polynomial and $z$ runs over the whole field $\F_{q^n}$, or over primitive elements. Further results in this context are developed in~\cite{dig4, dig5}. Most notably, in~\cite{dig4} the author explores the existence of polynomial values $P(z)$ with $k$ prescribed digits $a_{i_1}, \ldots, a_{i_k}\in \F_{q}$, where $1\le i_1<\cdots<i_k\le n$ and the argument $z$ runs over the whole field $\F_{q^n}$, or over the set of primitive elements. The main results there, Theorems 1.1 and 1.6 show that we can asymptotically reach the interval $k\in [1, n/2)$. Equivalently, one can prescribe up to around half of the digits of polynomial values and polynomial values with primitive arguments. For the special case $P(z)=z$ with $z$ running over the set of primitive elements, we are simply counting primitive elements with prescribed digits. In this direction, we have a nice improvement of the previous result.

\begin{proposition}\label{prop-2}
For each positive integer $n\ge 2$, let $p_n$ be its smallest prime factor. Then there exists a constant $c(n)>0$ such that, for every $q>c(n)$ and every basis $\mathcal B=\{b_1, \ldots, b_n\}$, there exists a primitive element of $\F_{q^n}$ with up to $n-n/{p_n}-1\ge n/2-1$ digits prescribed with respect to $\mathcal B$. 
\end{proposition}

\begin{proof}
By the definition, the set of elements in $\F_{q^n}$ with $k$ digits prescribed $a_{i_1}, \ldots, a_{i_k}\in \F_{q}$ comprises an $\F_q$-affine space of dimension $n-k$. If $k\le n-n/{p_n}-1$, we have that $n-k\ge n/p_n+1$ and the result follows by a similar argument employed in Proposition~\ref{prop:grass}. We omit details.
\end{proof}
We emphasize that the bound $n-n/{p_n}-1$ in Proposition~\ref{prop-2} is sharp for arbitrary bases. In fact, let $\mathcal B_0=\{b_1, \ldots, b_{n/p_n}\}$ be an $\F_q$-basis for the field $\F_{q^{n/p_n}}$ and let $\mathcal B_1=\{b_1, \ldots, b_{n/p_n}, b_{n/p_n+1}, \cdots, b_{n}\}$ be any completion to an $\F_q$-basis for $\F_{q^n}$. In particular, if we prescribe the $n-n/p_n$ digits $$a_{n/p_n+1}=\cdots=a_{n}=0,$$ the corresponding elements lie in $\F_{q^{n/p_n}}$, hence none of them can be a primitive element of $\F_{q^n}$.

\subsection{On primitive $k$-normal elements}
For $\beta\in \F_{q^n}$, the elements $\beta, \beta^q, \ldots, \beta^{q^{n-1}}$ are the $\F_q$-Galois conjugates of $\beta$. The element $\beta$ is normal over $\F_q$ if their conjugates comprise an $\F_q$-basis for $\F_{q^n}$. The celebrated Primitive Normal Basis Theorem (PNBT) ensures the existence of normal elements that are also primitive for every finite field extension. Its first proof was given by Lenstra and Schoof~\cite{lenstra} an a free computer proof was later given by Cohen and Huczynska~\cite{cohen}. Following the concept of normal elements, Huczynska et al~\cite{HMPT} introduced the notion of $k$-normal elements. These are the elements  $\beta\in \F_{q^n}$ for which the $\F_q$-Galois conjugates $\beta, \beta^q, \ldots, \beta^{q^{n-1}}$ generate an $\F_q$-vector space of dimension $n-k$. In this context, $0$-normal elements are the original normal elements and $0\in \F_{q^n}$ is the unique $n$-normal element. Motivated by the PBNT, they proposed a challenging problem (see Problem 6.3 in~\cite{HMPT}).

\begin{problem}\label{pr}
Determine the pairs $(n, k)$ such that there exist primitive $k$-normal elements in $\F_{q^n}$ over $\F_q$.
\end{problem}

Before discussing Problem~1, let us provide basic information on $k$-normal elements, which can be found in~\cite{HMPT}.

\begin{lemma}\label{lem:basic}
For each element $\alpha\in \F_{q^n}$, the set of polynomials $$g(x)=\sum_{i=0}^ta_ix^i\in \F_q[x],$$ such that $0=g\circ \alpha:=\sum_{i=0}^ta_i\alpha^{q^i}$ is an ideal of $\F_q[x]$. This ideal is generated by a monic polynomial $m_{\alpha, q}(x)$, the $\F_q$-order of $\alpha$. Moreover, the following hold:
\begin{enumerate}[(i)]
\item $m_{\alpha, q}$ is a divisor of $x^n-1$;
\item $\alpha$ is $k$-normal over $\F_q$ if and only if $m_{\alpha, q}$ has degree $n-k$;
\item for each monic divisor $g\in \F_q[x]$ of $x^n-1$, there exist $\Phi_q(g)$ elements $\alpha\in \F_{q^n}$ such that $m_{\alpha, q}=g$. Here $\Phi_q(g)$ denotes the polynomial analogue for the Euler totient function. This function satisfies $$(q-1)^{\deg(g)}\le \Phi_q(g)\le q^{\deg(g)}-1.$$
\end{enumerate}
\end{lemma}

In particular, the existence of $k$-normal elements in $\F_{q^n}$ over $\F_q$, without any restriction of being primitive, is conditioned to the existence of an $(n-k)$-degree monic divisor of $x^n-1$ over $\F_q$. As pointed out in~\cite{HMPT}, the only values of $k$ for which this is ensured for generic $(q, n)$ are $k=0, 1, n-1, n$. Therefore  Problem~\ref{pr} should be reformulated, including the parameter $q$.

By providing a simple bound on the multiplicative order of $(n-1)$-normal elements, the authors in~\cite{HMPT} show that no $(n-1)$-normal element can be primitive if $n>1$. Therefore, we should consider Problem~\ref{pr} with $k$ in the range $[0, n-2]$. On the other hand, the case $k=0$ is the PBNT and the case $k=1$ yield positive answer to Problem~\ref{pr}, under the necessary condition $n\ge 3$~\cite{RT18}. This is not surprise, since $k$-normal elements are quite abundant in $\F_{q^n}$ if $k=0$ or $k=1$ (see item (iii) of Lemma~\ref{lem:basic}).

In~\cite{R}, the author discusses Problem~\ref{pr} under the natural condition that there exist $k$-normal elements in $\F_{q^n}$ over $\F_q$. With this assumption, he gives positive answer to Problem~\ref{pr} with $k$ in the range $[0, n/2(1-\varepsilon_{q, n})]$, where $\varepsilon_{q, n}\to 0$ if $q\to+\infty$ or $n\to+\infty$. In particular, if $n$ is fixed, $k<n/2$ and $q$ is sufficiently large, there exist primitive $k$-normal elements whenever $k$-normal elements actually exist. The extreme $k=n/2$ have shown to be a genuine exception when $n=4$ and $q\equiv 3\pmod 4$. See Proposition 3.4 in~\cite{R} and the comments thereafter for more details. 

We observe that, in general, Problem~\ref{pr} contains the following seemingly hard sub problem. 
\begin{problem}\label{pr2}
Given $k$, determine the pairs $(q, n)$ with $n\ge k$ such that $x^n-1$ has a divisor of degree $k$, defined over $\F_q$.
\end{problem}
An integer $n$ is called $\F_q$-practical if $x^n-1$ has divisors over $\F_q$ of every possible degree $0\le k\le n$. When $q=p$ is a prime, we have estimates on the number of $\F_p$-practical numbers up to $x$~\cite{PTW, T13}, where the result in~\cite{T13} is conditioned to the GRH. These seem to be the most significant progress on Problem~\ref{pr2}. A more detailed discussion on Problem~\ref{pr2} is provided in~Section~5 of~\cite{R}, where an infinite family of $\F_q$-practical numbers is provided.

We aim to improve results in~\cite{R} in the context of Problem~\ref{pr}, taking into account the hardness of  Problem~\ref{pr2}. We first discuss a more restricted assumption that we have to impose on $k$. The proof of the non existence of primitive $(n-1)$-normal elements in~\cite{HMPT} is given as follows. If $n>1$ and $\alpha\in \F_{q^n}^*$ is $(n-1)$-normal, $m_{\alpha, q}$ has degree one and so it is of the form $x-\delta$ with $\delta\in \F_{q}^*$. Therefore, we have that $$0=(x-\delta)\circ \alpha=\alpha^q-\delta \alpha.$$ In particular, $\alpha^{q-1}=\delta$ and so $\alpha^{(q-1)^2}=1$. The latter entails that the multiplicative order of $\alpha$ is at most $(q-1)^2<q^n-1$, hence $\alpha$ cannot be primitive. This idea, that was also employed in~\cite{R} for $(n, k)=(4, 2)$, is easily extended in the following lemma.

\begin{lemma}\label{lem:basic2}
If $\alpha\in \F_{q^n}$ is such that $m_{\alpha, q}$ divides a binomial $x^t-\delta\in \F_q[x]$ with $1\le t<n$, then $\alpha$ is not a primitive element of $\F_{q^n}$.
\end{lemma}
\begin{proof}
If $m_{\alpha, q}$ divides $x^t-\delta$ and $\delta\in \F_{q}^*$, we have that $0=(x^t-\delta)\circ \alpha=\alpha^{q^t}-\delta \alpha$ and so $\alpha^{(q^t-1)(q-1)}=1$. However, $(q^t-1)(q-1)<q^n-1$ if $t<n$. The latter implies that $\alpha$ cannot be a primitive element of $\F_{q^n}$.
\end{proof}

Lemmas~\ref{lem:basic} and~\ref{lem:basic2} motivate us to introduce the following definition.

\begin{definition}
An element $\alpha\in \F_{q^n}$ is free of binomials if $f(x)=x^n-1$ is the unique monic binomial  $f\in \F_q[x]$ of degree at most $n$ such that $f\circ \alpha=0$.
\end{definition}

In the context of Problem~\ref{pr}, Lemma~\ref{lem:basic2} entails that the existence of $k$-normal elements that are free of binomials is necessary. In the following theorem we prove that, for $q$ sufficiently large, this condition is also sufficient.

\begin{theorem}\label{thm:k-normal}
Let $n\ge 2$ be a positive integer. Then there exists a constant $c(n)>0$ such that, for every $q>c(n)$ and every $0\le k\le n-2$, the following are equivalent:

\begin{enumerate}[(i)]
\item there exists a $k$-normal element in $\F_{q^n}$ over $\F_q$ that is free of binomials;

\item there exists a $k$-normal element in $\F_{q^n}$ over $\F_q$ that is primitive.
\end{enumerate}
\end{theorem}

\begin{proof}
The direction (ii)$\rightarrow$(i) follows by Lemma~\ref{lem:basic2}. For the direction (i)$\rightarrow$(ii), suppose that there exists a $k$-normal element free of binomials in $\F_{q^n}$. In particular, there exists an $(n-k)$-degree monic divisor $g(x)=\sum_{i=0}^{n-k}a_ix^i\in \F_q[x]$ of $x^n-1$ such that $g$ does not divide any binomial $x^t-\delta$ with $t<n$ and $\delta\in \F_q$. Let $\alpha\in \F_{q^n}^*$ be an element whose $\F_q$-order equals $m_{\alpha, q}(x)=g(x)$ and let $\mathcal V_{g}$ be the $(n-k)$-dimensional $\F_q$-vector space comprising the roots of 
$L_g(x):=\sum_{i=0}^{n-k}a_ix^{q^i}$.

It follows by the definition that the set $\{\alpha, \ldots, \alpha^{q^{n-k-1}}\}$ comprises an $\F_q$-basis for $\mathcal V_g$, and so $\mathcal V_g\subseteq \F_{q^n}$. We claim that $\alpha^{q-1}=\alpha^q\cdot \alpha^{-1}$ has degree $n$ over $\F_q$. In fact, if $\alpha^{q-1}$ has degree $d$ over $\F_q$ with $d<n$, we have that $\alpha^{(q-1)(q^{d}-1)}=1$, i.e., $\alpha^{q^d-1}=\delta\in \F_q^*$. In other words, $(x^d-\delta)\circ \alpha=0$. Lemma~\ref{lem:basic} entails that $g(x)=m_{\alpha, q}(x)$ divides $x^d-\delta$, a contradiction with the initial assumption on $g(x)$. Therefore, we are under the conditions of Proposition~\ref{prop:main}.  In particular, there exists a constant $\kappa>0$ such that $\mathcal V_{g}$ contains at least $q^{n-k-1/2}$ primitive elements for $q>\kappa$. From Lemma~\ref{lem:basic}, any element $\beta\in \mathcal V_{g}$ with $m_{\beta, q}(x)\ne g(x)$ must satisfy $m_{\beta, q}(x)=h(x)$, where $h(x)\in \F_q[x]$ is a monic divisor of $g(x)$ with degree at most $(n-k)-1$. In particular, from item (iii) in Lemma~\ref{lem:basic}, the number of such $\beta$'s is at most $q^{n-k-1}\cdot d$, where $d$ is the number of distinct monic divisors of $g(x)$ that are defined over $\F_q$. From the trivial bound $d\le 2^{n-k}$, there exists a primitive element $\alpha_0\in \mathcal V_g$ with $m_{\alpha_0, q}(x)=g(x)$ provided that
$$q^{n-k-1/2}-2^{n-k}q^{n-k-1}=q^{n-k-1}(q^{1/2}-2^{n-k})>0.$$
It suffices to require that $q>4^n$. In this case, such $\alpha_0\in \F_{q^n}$ is a primitive $k$-normal element over $\F_q$.
\end{proof}

\subsubsection{Primitive elements with low normality}
As previously mentioned, the existence of primitive $k$-normal elements for $k=0, 1$ is naturally expected, since these elements appear with high frequency in $\F_{q^n}$. However, an interesting conjecture on Artin-Schreier extensions suggests that we may have counter intuitive configurations. If $q=p$ is a prime and $a\in \mathbb\F_{p}^*$ is a primitive element, the Artin-Schreier polynomial $x^p-x-a\in \F_p[x]$ is irreducible. Therefore, for any root $\theta$ of such polynomial, we have that $\F_{p^p}=\F_p(\theta)$ and the other roots of the same polynomial are the translates $\theta+a, a\in \F_p$. A remarkable conjecture in the theory of finite fields is that $\theta$ is a primitive element of $\F_{p^p}$, i.e., its multiplicative order equals $p^p-1$~\cite{W}. The results of Carlitz-Davenport~\cite{C, D} on extensions $\F_{q^n}/\F_q$ with the translate property cannot be employed since they are asymptotic and generically require that $n<q$; this fails for $n=q=p$. Some study on the multiplicative order of $\theta$ has been made, including arithmetic constraints~\cite{Mont} and lower bounds~\cite{V04}. Since $\theta^p-\theta=a\in \F_p^*$, we have that $m_{\theta, p}(x)=(x-1)^2$ and so $\theta$ is $(p-2)$-normal over $\F_p$. However, since $x^p-1=(x-1)^p$, the $\F_q$-order of every $(p-2)$-normal element equals $(x-1)^2$. In particular, they are roots of the polynomial $x^{p^2}-2x^p+x$ and so the number of such elements is at most $p^2$. The quantity $p^2$ is extremely small when compared to $p^p$ if $p$ is large.

Although such conjecture remains open, we can extend its setting to a similar counter intuitive situation where Theorem~\ref{thm:k-normal} applies. This is shown in the following corollary.

\begin{corollary}
Let $q$ be a power of a prime $p$. Then there exists a constant $c=c(p)>0$ such that, if $q>c$, $\F_{q^p}$ contains primitive $k$-normal elements for every $0\le k\le p-2$.
\end{corollary}
\begin{proof}
We observe that $x^p-1=(x-1)^p$ and so there exist $k$-normal elements for every $0\le k\le p$. In fact, for any $k$-normal element $\theta\in \F_{q^p}$, $m_{\theta, q}(x)=(x-1)^{p-k}$. Since $p$ is prime, for every $0\le k\le p-2$, the polynomial $(x-1)^{p-k}$ divides a binomial $x^s-\delta\in \F_q[x]$ with $s\le p$ if and only if $p=s$ and $\delta=1$. In other words, for $0\le k\le p-2$, every $k$-normal element in $\F_{q^p}$ is free of binomials. The result follows from Theorem~\ref{thm:k-normal}.
\end{proof}

\end{document}